\documentclass{amsart}

\usepackage{graphicx}
\usepackage{amssymb}
\usepackage{fancyvrb}

\usepackage{listings}
\lstdefinelanguage{GAP}{%
  morekeywords={%
    Assert,Info,IsBound,QUIT,%
    TryNextMethod,Unbind,and,break,%
    continue,do,elif,%
    else,end,false,fi,for,%
    function,if,in,local,%
    mod,not,od,or,%
    quit,rec,repeat,return,%
    then,true,until,while%
  },%
  sensitive,%
  morecomment=[l]\#,%
  morestring=[b]",%
  morestring=[b]',%
}[keywords,comments,strings]

\usepackage[T1]{fontenc}
\usepackage[variablett]{lmodern}
\usepackage{xcolor}
\newtheorem{theorem}{Theorem}
\newtheorem{proposition}[theorem]{Proposition}

\newtheorem{corollary}[theorem]{Corollary}

\theoremstyle{definition}

\numberwithin{equation}{section}

\DeclareMathOperator{\GL}{GL}
\DeclareMathOperator{\PSL}{PSL}

\DeclareMathOperator{\Aut}{Aut}

\DeclareMathOperator{\Homeo}{Homeo}
\DeclareMathOperator{\Comm}{Comm}
\DeclareMathOperator{\BS}{BS}
\DeclareMathOperator{\QI}{QI}

\DeclareMathOperator{\Z}{\mathbb{Z}}

\newcommand{\setU}

\def\R{\mathbb{R}}
\def\C{\mathbb{C}}

\def\Z{\mathbb{Z}}

\def\<{\langle}
\def\>{\rangle}

\begin{document}

\title[On abstract commensurators of surface groups]{On abstract commensurators of surface groups}

\author{Khalid Bou-Rabee}
\address{Department of Mathematics, The City College of New York}
\email{khalid.math@gmail.com}

\author{Daniel Studenmund}
\address{University of Notre Dame}
\email{dstudenm@nd.edu}
\thanks{DS supported in part by NSF Grant \#1547292.}


\date{March 8, 2019}


\begin{abstract}
Let $\Gamma$ be the fundamental group of a surface of finite type and $\Comm(\Gamma)$ be its abstract commensurator.
Then $\Comm(\Gamma)$ contains the solvable Baumslag--Solitar groups $\langle a ,b : a b a^{-1} = b^n \rangle$ for any $n > 1$.
Moreover, the Baumslag--Solitar group $\langle a ,b : a b^2 a^{-1} = b^3 \rangle$ has an image in $\Comm(\Gamma)$ that is not residually finite.
Our proofs are computer-assisted.

Our results also illustrate that finitely-generated subgroups of $\Comm(\Gamma)$ are concrete objects amenable to computational methods. For example, we give a proof that $\langle a ,b : a b^2 a^{-1} = b^3 \rangle$ is not residually finite without the use of normal forms of HNN extensions.
\end{abstract}


\maketitle

\section{Introduction}

Let $G$ be a group. The {\em abstract commensurator}\footnote{In the literature this is also known as the \emph{virtual automorphism group}. } of $G$, denoted $\Comm(G)$, is the set of equivalence classes of isomorphisms $\phi: H_1 \to H_2$ for finite-index subgroups $H_1,H_2\leq G$, where two isomorphisms are equivalent if they are both defined and equal on a common finite-index subgroup of $G$. The set $\Comm(G)$ is a group under composition. Elements of $\Comm(G)$ are called \emph{commensurators} of $G$.

The results of this paper serve two purposes. 
First, they shed light on the structure of $\Comm(\Gamma)$ when $\Gamma$ is the fundamental group of a surface of finite type. 
These abstract commensurator groups are not well-understood, and in particular are known to be neither finitely generated \cite{MR2736164} nor linear over $\C$ \cite[Proposition 7.6]{MR3351746}. 
Second, the methods of proof suggest that maps into abstract commensurators of fundamental groups of surfaces of finite type provide a new sort of ``representation theory'' of groups that have no faithful representations through matrices. 
This view has immediate utility: we give a new and concrete proof of a classical result of Baumslag--Solitar concerning their one-relator groups \cite{MR0142635}.



\subsection{The main result} \label{ssec:mainresult}
The \emph{Baumslag--Solitar groups} are given by the presentation 
$$\BS(m,n) := \left<a ,b : a b^m a^{-1} = b^n \right>.$$
By {\em surface group}, we mean the fundamental group of a surface with finite genus, and finitely many punctures and boundary components. Any surface group is finitely generated.

\begin{theorem} \label{thm:main}
Let $\Gamma$ be a surface group, and let $n > 1$.
Then the group $\Comm(\Gamma)$ contains $\BS(1,n)$.
Moreover, $\Comm(\Gamma)$ contains a non-residually finite image of $\BS(2,3)$.
\end{theorem}

Our proof of Theorem \ref{thm:main} appears in \S \ref{sec:theproofs} with supplementary code and theory appearing in the appendix.
Since $\Comm(G)$ is an invariant of the commensurability class of $G$, Theorem \ref{thm:main} is a statement about the abstract commensurators of two groups: the nonabelian free group of rank two, and the fundamental group of a compact surface of genus two. 

We prove the nonabelian free group case first, and then build upon it to handle the compact case.
The first part of each case uses some of the elementary theory of surface topology. 
The second part of each case requires some computer-assistance, using the GAP System for Computational Discrete Algebra \cite{GAP4}.

The length of our proof belies the difficulty of the proof of the cocompact case, which relies on a long and delicate computation supplied in Appendix \ref{app:surfacegroup}.
The appendix is a substantial part of the theory in this paper: the  setup, hand computations, and computer computations find an explicit representation of a Baumslag--Solitar group in $\Comm(\Gamma)$.

\subsection{On local residual finiteness} \label{ssec:localresfin}
A group is \emph{(locally) residually finite} if every (finitely generated) subgroup is residually finite. Linear groups (e.g., fundamental groups of surfaces, $\GL_n(\C)$, or $\GL_n(\mathbb{F}_p[t])$) and mapping class groups of surfaces (e.g., braid groups) are important classes of  locally residually finite groups \cite{M56,MR0405423}.

Our main theorem determines precisely when the  abstract commensurator of a lattice in a semisimple Lie group is locally residually finite.
See \S \ref{sec:proofanswer} for the proof.

\begin{corollary} \label{cor:answer}
Let $\Gamma$ be an irreducible lattice in a semisimple Lie group $G$ without compact factors.
Then $\Comm(\Gamma)$ is locally residually finite if and only if $G$ is not locally isomorphic to $\PSL_2(\R)$.
In particular, the abstract commensurator of any lattice in $\PSL_n(\R)$ is locally residually finite if and only if $n \neq 2$.
\end{corollary} 

\noindent
The fact that lattices in $\PSL_2(\R)$ have abstract commensurators which are not locally residually finite is surprising for two reasons.
First, Baumslag showed that for any finitely-generated residually finite $G$, the group $\Aut(G)$ (which $\Comm(G)$ can be thought to generalize) is residually finite \cite{MR0146271}. Second, Odden showed that the abstract commensurator of the fundamental group of a closed surface is a mapping class group of an inverse limit of surfaces (the ``universal hyperbolic solenoid'') \cite{MR2115078}.

\subsection{On residual finiteness of $\BS(2,3)$} \label{ssec:britton}
Our proof of Theorem \ref{thm:main} explicitly identifies an element $\gamma \in \BS(2,3)$ in the kernel of a surjective homomorphism $\BS(2,3) \to \BS(2,3)$, then concretely shows that $\gamma$ has nontrivial image under a homomorphism $\BS(2,3) \to \Comm(\Gamma)$. 
Thus, our methods prove that $\BS(2,3)$ is not residually finite without using the normal forms ensured by the Britton's Lemma \cite{MR0168633}. Bypassing this step makes the original proof that $\BS(2,3)$ is not residually finite, due to Baumslag--Solitar \cite{MR0142635}, elementary and concrete. We emphasize that our proof relies solely on the $\Comm(F_2)$ case, and hence does not require computer verification.
See \S \ref{sec:BSnotRF} for the proof.

\begin{theorem}[\cite{MR0142635}] \label{thm:BSnotRF}
The group $BS(2,3)$ is not residually finite.
\end{theorem}

\subsection{On two groups} \label{ssec:twogroups}
Let $F_2$ be the nonabelian free group of rank two. Let $\Gamma_2$ be the fundamental group of a compact surface of genus two.
The groups $\Comm(F_2)$ and $\Comm(\Gamma_2)$ share a number of properties: Bartholdi--Bogopolski showed that neither group is finitely generated \cite{MR2736164}. Moreover, finitely-generated subgroups of them have solvable word problem (see \S \ref{ssec:wordproblem}). The proof of Theorem \ref{thm:main} shows that they both contain many infinite images of Baumslag--Solitar groups.

In spite of these similarities, in \S \ref{sec:prooftwogroups}  we show that $\Comm(F_2)$ contains more finite subgroups than $\Comm(\Gamma_2)$ to prove that they are not isomorphic.

\begin{proposition} \label{prop:twogroups}
The groups $\Comm(F_2)$ and $\Comm(\Gamma_2)$ are not isomorphic.
\end{proposition}

\subsection{On the word problem} \label{ssec:wordproblem}
Although the groups $\Comm(\Gamma)$ are not well-understood when $\Gamma$ is a surface group, their finitely-generated subgroups may be understood concretely, as in the computations that appear in Appendix \ref{app:freegroup}. In particular, the word problem is solvable. A more general version of the following is proved in \S \ref{sec:partingquestion}.

\begin{proposition}
Finitely-generated subgroups of abstract commensurators of surface groups have solvable word problem.
\end{proposition}

\noindent
Theorem \ref{thm:main} provides an example of a non-residually finite group $G$ that faithfully embeds in $\Comm(F_2)$. Such $G$ cannot be completely understood through its linear representations, while words in its image in $\Comm(F_2)$ can be evaluated by a computer.
This shows that representations into $\Comm(F_2)$ may be useful for understanding abstract groups without faithful linear representations.
See \S \ref{sec:partingquestion} for further discussion and questions.

\subsection{Historical remarks} 
\label{ssec:somehistory}
Questions about abstract commensurators of surface groups have been asked for at least 30 years \cite{Mann87}. More generally, abstract commensurators detect arithmeticity of lattices in higher-rank semisimple Lie groups (see discussion after Corollary 2 in \S \ref{ssec:localresfin}), give ``generalized Hecke operators'' from number theory (see \cite[page 10]{MR1767270}), play a role in extending linear group methods to nonlinear settings (e.g., with mapping class groups \cite[\S 1]{MR2354204}), and are used to give more algebraic descriptions of existing infinite finitely presented simple groups (e.g., \cite{MR1950873}).

Abstract commensurators can also be understood in the context of coarse geometry. The \emph{quasi-isometry group} $\QI(\Gamma)$ of a group $\Gamma$ is the group of equivalences classes of self-quasi-isometries $f:\Gamma\to\Gamma$, up to coarse equivalence. There is a natural map  $\Comm(\Gamma) \to \QI(\Gamma)$ which is injective when $\Gamma$ is finitely generated \cite[Theorem 7.4]{MR1937831}. This shows, for example, that the abstract commensurator of a word hyperbolic group acts faithfully on its Gromov boundary. 

That the abstract commensurator of a free group of rank two is not locally residually finite is a well-known folklore result.
The folklore proof relies on the existence of a simple, finitely presented group that is isomorphic to an amalgamated product of two finitely-generated free groups over a common finite-index subgroup. The first groups of this kind were constructed by Burger--Mozes  \cite{MR1446574}. 
It is unknown whether any of the Burger--Mozes groups embed inside the abstract commensurator of a closed oriented surface group of genus two. 

\subsection*{Acknowledgements}
We are grateful to Benson Farb and Andrew Putman for conversations and support. We are especially grateful to Yves Cornulier and Benson Farb for giving us many useful comments that substantially improved the paper.

\section{Proof of Theorem \ref{thm:main}} \label{sec:theproofs}

Let $F_2$ be the free group of rank two. Let $\Sigma_g$ be the compact oriented surface of genus $g$. 
Let $\Gamma$ be a surface group. As stated in \S \ref{ssec:mainresult} and \S \ref{ssec:twogroups}, the group $\Comm(\Gamma)$ is isomorphic to either $\Comm(F_2)$ or $\Comm(\pi_1(\Sigma_2,*))$.
We handle each case separately, as they are different enough to warrant different proofs (the first serving as an outline of the second).

It will be important to us in each case that $\Gamma$ satisfies the {\em unique root property}: if elements $x,y \in \Gamma$ satisfy $x^n = y^n$ for some $n\neq 0$, then $x=y$ \cite[Lemma 2.2]{MR2736164}, \cite[pages 462--463]{MR1744486}. 
A consequence of the unique root property is that two isomorphisms $f:A\to B$ and $g:C\to D$ between finite-index subgroups of $\Gamma$ represent the same element of $\Comm(\Gamma)$ if and only if $f\mid_{A\cap C} = g\mid_{A\cap C}$.
In particular, to check that $f$ represents a nontrivial elements of $\Comm(\Gamma)$, it suffices to find an element $\gamma\in \Gamma$ such that $f(\gamma) \neq \gamma$.

\subsection{The $\Comm(F_2)$ case} \label{ssec:proofoffreecase}
We first describe a method for constructing images of $\BS(n,m)$ in $\Comm(F_2)$. Set $F_2 = \left< A, B \right>$.
Let $\pi_1 : F_2 \to \Z/m \times \Z/n$ be the map given by $A \mapsto (1,0)$ and $B \mapsto (0,1)$.
Let $\pi_2 : F_2 \to \Z/m \times \Z/n$ be the map given by $A \mapsto (0,1)$ and $B \mapsto (1,0)$.
Set $\Delta_1 = \ker(\pi_1)$ and $\Delta_2 = \ker(\pi_2)$.

Let $\phi$ be the commensurator with representative
$f: F_2 \to F_2$ given by $f(\gamma) = A \gamma A^{-1}$.
Let $\psi$ be a commensurator with representative $g: \Delta_1 \to \Delta_2$, where $g(A^m) = A^n$; such an isomorphism $g$ exists because $\Delta_1$ and $\Delta_2$ are free groups of the same rank in which $A^m$ and $A^n$, respectively, are elements in free generating sets.
Then the commensurator $\psi \circ \phi^m \circ \psi^{-1}$ has a representative $h = g \circ f^m \circ g^{-1}$ such that for every $\gamma \in \Delta_2$,
$$
h(\gamma) = g \circ f^m \circ g^{-1} (\gamma)
= g( A^{m} g^{-1}(\gamma) A^{-m})
= A^n (g \circ g^{-1}(\gamma)) A^{-n}
= A^n \gamma A^{-n},
$$
and hence $\psi \circ \phi^m \circ \psi^{-1} = \phi^n$.
Thus, the assignment $a \mapsto \psi$ and $b \mapsto \phi$ defines a homomorphism $\Phi: \BS(m,n) \to \Comm(F_2)$.

Next, we show that the map $\BS(m,n) \to \left< \psi, \phi \right>$ is an isomorphism of groups when $m = 1$.
Let $z$ be in the kernel of this map. Then $z = a^s \gamma$, where $\gamma$ is in the normal closure of $b$.
Since $\left< \left< b \right> \right> = \left< a^t b a^{-t} : t \leq 0 \right>$, we have that $z$ is conjugate to an element of the form $a^s b^t$ for some integers $s$ and $t$.
It suffices to show that $\psi^s \circ \phi^t$ is only trivial in $\Comm(F_2)$ if $s=t=0$.

If $s>0$ then any $\psi^s$ has a representative that maps $A$ to $A^{n^s}$.
By the unique root property of $F_2$, it is impossible for an automorphism of $F_2$ to send $A$ to $A^{n^s}$.
Similarly, if $s<0$ then $\psi^s$ has a representative that maps $A^{n^{-s}}$ to $A$, which cannot be extended to an automorphism of $F_2$.
If $\psi^s \circ \phi^t$ is trivial, then $\psi^s$ has a representative that is an automorphism of $F_2$, and so $s=0$.
It is clear that $\phi^t$ is trivial only if $t=0$ because $f^t(B^\ell) = A^t B^\ell A^{-t}$ for all $\ell$. This complete the proof of the $m=1$ case.


To finish, we need to show that when $m = 2$ and $n = 3$, the image of $\BS(m,n)$ as defined above is not residually finite.
To do this, we need to show that element
$$
\gamma := b^{-1} a b a^{-1} b^{-1} a b a^{-1} b^{-1}
$$
has nontrivial image.
Indeed, it is in the kernel of the surjective map $\BS(2,3) \to \BS(2,3)$ given by $a \mapsto a$ and $b \mapsto b^2$, and thus is in the residual finiteness kernel of $\BS(2,3)$. See \S \ref{sec:BSnotRF} for more details.

The proof is completed by evaluating the commensurator $\Phi(\gamma)$ on the word $BAB^{-1}A^{-1}\in F_2$ and checking that the result is not equal to $BAB^{-1}A^{-1}$. This is straightforward, but tedious, to compute by hand.
To perform this calculation, we used computer assistance. See the code in Appendix \ref{app:freegroup} and explanations there to end the proof of the $\Comm(F_2)$ case.

\subsection{The $\Comm(\pi_1(\Sigma_2,*))$ case} \label{ssec:surfacecaseproof}

As in \S \ref{ssec:proofoffreecase}, we begin by describing how to obtain images of $\BS(n,m)$ inside $\Comm(\pi_1(\Sigma_2,*))$.
Let $[X,Y] := XY X^{-1} Y^{-1}$ and let
$$\Gamma_2 := \pi_1(\Sigma_2,*) = \left< A, B, C, D : [A,B][C,D] \right>.$$
Set $\pi_1 : \Gamma_2 \to \Z/m \times \Z/n$ be a map satisfying $A \mapsto (1,0)$, where in the cover corresponding to $\ker\pi_1$ the curve corresponding to $A^m$ lifts to a non-separating simple closed curve. 
See Figure \ref{fig:maincover} for the cover corresponding to an example of such a map.
Similarly, set $\pi_2 : \Gamma_2 \to \Z/m \times \Z/n$ be a map satisfying  $A \mapsto (0,1)$, where in the cover corresponding to $\ker\pi_1$ the curve corresponding to $A^n$ lifts to a non-separating simple closed curve.
Set $\Delta_1 = \ker(\pi_1)$ and $\Delta_2 = \ker(\pi_2)$.

\begin{figure}
  \centering
  \includegraphics[width=10cm]{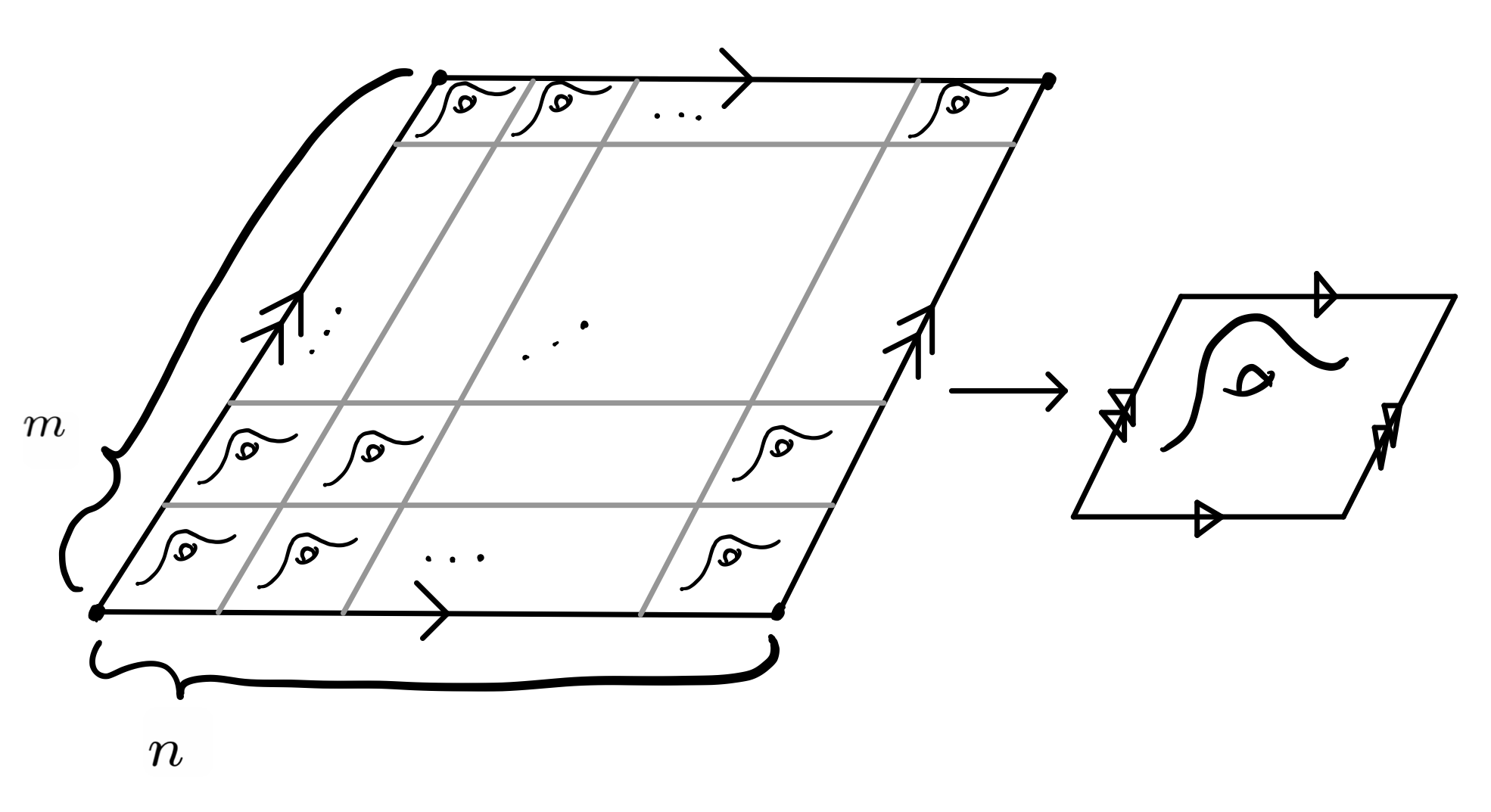}
  \caption{Cover of $\Sigma_2$ for \S \ref{ssec:surfacecaseproof}}
  \label{fig:maincover}
\end{figure}

Let $\phi$ be the commensurator with representative
$f: \Gamma_2 \to \Gamma_2$ given by $f(\gamma) = A \gamma A^{-1}$.
To define the commensurator $\psi$ we need some additional setup:

Let $p_1: S_1 \to \Sigma_2$ and $p_2:S_2 \to \Sigma_2$ be the covers corresponding to $\Delta_1$ and $\Delta_2$ respectively. Then $A^m$ and $A^n$ lift to  non-separating simple-closed curves in $S_1$ and $S_2$ by construction.
Any oriented surface group is uniquely determined by its Euler characteristic. Both $S_1$ and $S_2$ cover $\Sigma_2$ with degree $mn$ and hence have Euler characteristic $-2mn$ by the Riemann--Hurwitz formula. It follows that $S_1$ and $S_2$ are homeomorphic.
Moreover, by the \emph{Change of Coordinates Principle} \cite[p. 37]{MR2850125} there is\footnote{See Appendix  \ref{app:surfacegroup} for an explicit construction of such a map in the case $(m,n) = (2,3)$.} a homeomorphism $S_1 \to S_2$ inducing an isomorphism $g: \Delta_1 \to \Delta_2$, where $g(A^m) = A^n$. Let $\psi$ be the commensurator with representative $g$.
Then the commensurator $\psi \circ \phi^n \circ \psi^{-1}$ has a representative $h = g \circ f^n \circ g^{-1}$ such that for every $\gamma \in \Delta_2$,
$$
h(\gamma) = g \circ f^m \circ g^{-1} (\gamma)
= g( A^{m} g^{-1}(\gamma) A^{-m})
= A^n (g \circ g^{-1}(\gamma)) A^{-n}
= A^n \gamma A^{-n},
$$
and hence $\psi \circ \phi^m \circ \psi^{-1} = \phi^n$.
Thus, the assignment $a \mapsto \psi$ and $b \mapsto \phi$ defines a homomorphism $\Phi: \BS(m,n) \to \Comm(\Gamma_2)$.

Since $\Gamma_2$ has the unique root property, the argument for showing that the map $\BS(m,n) \to \left< \psi, \phi \right>$ is an isomorphism when $m=1$ in \S \ref{ssec:proofoffreecase} applies here verbatim.
Hence, it only remains to show that when $m = 2$ and $n = 3$ the element
$$
\gamma := b^{-1} a b a^{-1} b^{-1} a b a^{-1} b^{-1}
$$
has nontrivial image, since this is in the residual finiteness kernel of $\BS(m,n)$.
The rest of the proof, as before, is computer-assisted. The computer computations are more difficult to implement here as surface groups do not have as much flexibility as free groups. See Appendix \ref{app:surfacegroup} and explanations there to end the proof. \qed

\section{Proof of Corollary \ref{cor:answer}}
\label{sec:proofanswer}

By Theorem \ref{thm:main} it suffices to show that the abstract commensurator of any irreducible lattice $\Lambda$ in a semisimple Lie group $G$ without compact factors and not locally isomorphic to $\PSL_2(\R)$ is locally residually finite.
Indeed, for any such $\Lambda \leq G$, strong rigidity results of Mostow--Prasad--Margulis \cite{MR1090825} show that every commensurator of $\Lambda$ extends to an automorphism of $G$. By the Borel density theorem \cite{MR0123639}, the induced map $\Comm(\Lambda) \to \Aut(G)$ is faithful. This shows $\Comm(\Lambda)$ is linear and thus locally residually finite by Mal\'cev's Theorem \cite{MR0349383}.  \qed

\section{Proof of Theorem \ref{thm:BSnotRF}}
\label{sec:BSnotRF}

In \S \ref{ssec:proofoffreecase} there is a concrete and elementary proof that the element
$$
\gamma := b^{-1} a b a^{-1} b^{-1} a b a^{-1} b^{-1}
$$
is not the identity in $BS(2,3)$. We emphasize that this proof does not require a computer (although it is comforting that the computation can be checked using one).
The salient property of this element is that it is in the kernel of the surjective map $\rho: \BS(2,3) \to \BS(2,3)$ given by $a \mapsto a$ and $b \mapsto b^2$. It is a standard argument that this shows $\BS(2,3)$ is not residually finite. We include it here for completeness:

If $\BS(2,3)$ were residually finite, then there would exist a finite group $Q$ and a homomorphism $f : \BS(2,3) \to Q$ such that $f(\gamma) \neq 1$.
For each $k$ there is an element $\gamma_k$ such that $\rho^k(\gamma_k) = \gamma$. Then $\gamma_k$ is not in the kernel of $f\circ \rho^k$, but lies in the kernel of $f\circ \rho^n$ for all $n\geq k$, and hence the maps $f \circ \rho^k : \BS(2,3) \to Q$ are all distinct for positive $k$. This is impossible since there are only finitely many maps from a finitely-generated group to a fixed finite group. \qed

\section{Proof of Proposition \ref{prop:twogroups}}
\label{sec:prooftwogroups}

We will show that the group $\Comm(F_2)$ contains an infinite collection of groups not in $\Comm(\Gamma_2)$. To start, we record that $\Comm(F_2)$ contains every finite group. Every finitely-generated free group $F_n$ for $n\geq 2$ is a finite-index subgroup of $F_2$, and because $F_2$ has the unique root property, the natural maps $\Aut(F_n) \to \Comm(F_2)$ are injective. Every finite group lies in some symmetric group, hence in $\Aut(F_n)$ for some $n$.

However, $\Comm(\Gamma_2)$ only contains finite groups that have a cyclic subgroup of index at most 2. 
Indeed, the Gromov boundary of $\Gamma_2$ is $S^1$, and so $\Comm(\Gamma_2)$ acts faithfully on $S^1$ as described in \S\ref{ssec:somehistory}. (This action was originally observed by Odden \cite[Proposition 4.5]{MR2115078}.)
Since any square and any commutator in $\Homeo(S^1)$ is orientation preserving, we see that $[\Homeo(S^1):\Homeo_+(S^1)] = 2$. Moreover,  all finite subgroups of $\Homeo_+(S^1)$ are cyclic \cite[Lemma 3.1]{MR3813208}, it follows that every finite subgroup of $\Comm(\Gamma_2)$ contains a cyclic subgroup of index at most 2.  \qed

\section{A parting proof with questions}
\label{sec:partingquestion}
\subsection{}
If a finitely-generated group $\Gamma$ has solvable word problem, then finitely-generated subgroups of $\Aut(\Gamma)$ have solvable word problem. 
The same is true of $\Comm(\Gamma)$ when $\Gamma$ is finitely presented and satisfies the unique root property. 
In particular, finitely-generated subgroups of abstract commensurators of surface groups have solvable word problem. 

\begin{proposition} \label{prop:compute}
Let $\Gamma$ be a finitely presented group with solvable word problem and the unique root property. 
Then any finitely generated subgroup of $\Comm(\Gamma)$ has solvable word problem.
\end{proposition}

\begin{proof}
Fix a finite set $S$ of elements in $\Comm(\Gamma)$.
We outline an algorithm that inputs a word in $S$ and outputs ``True'' if the word is the identity, and ``False'' otherwise.
We take as constants in our algorithm representatives for each element in $S$, along with indices and finite generating sets for each of the domains and codomains of elements in $S$ (such generating sets exist by the Reidemeister-Schreier Method).
Given the above constants, the inverse of each element in $S$ can be computed, so without loss of generality we assume that $S$ is symmetric.

We start with the simplest case, when the word is of length one. Let $g : A \to B$ be the representative of the commensurator of $\Gamma$ corresponding to the letter, where $A, B$ are of finite index inside $\Gamma$ and have explicitly written generating sets. 
The isomorphism $g$ is fully determined by its values on the fixed finite generating set for $A$. Since the word problem in $\Gamma$ is solvable, determining whether $g$ fixes each generator of $A$ is solvable, and hence determining whether $g$ is the identity is solvable. It follows, by our assumptions on $\Gamma$, that determining whether the commensurator induced by $g$ is the identity is solvable.

For a word of length two we have two maps $g : A \to B$ and $h : C \to D$, where $A, B, C,D$ are all of finite index inside $\Gamma$ and have explicitly written generating sets and indices. Define
$$
g \star h = g |_{A \cap D} \circ h |_{h^{-1}(A \cap D)} 
$$
which is a map $h^{-1}(A \cap D) \to g(A \cap D)$ representing the composition of commensurators $[g]\circ [h] \in \Comm(\Gamma)$. Because $\Gamma$ has the unique root property, $[g]\circ [h]$ is trivial in $\Comm(\Gamma)$ if and only if $g\star h$ is the identity function.
Since $\Gamma$ is finitely presented, there is an algorithm for finding a finite quotient of $\Gamma$ whose kernel is contained in $A \cap D$.
Thus, there is an algorithm for computing a finite generating set for $A \cap D$, and subsequently $h^{-1}(A \cap D)$. As the homomorphism $g \star h$ is determined by its values on a generating set of $h^{-1}(A \cap D)$, determining whether $g \star h$ is the identity is solvable because $\Gamma$ has solvable word problem.

Repeating this process inductively for arbitrary compositions $h_1 \star h_2 \star h_3 \star \cdots \star h_k$ for letters $h_i$, allows one to handle word of arbitrarily length, giving the desired algorithm.
\end{proof}

\subsection{} The above proposition motivates the following questions:
\begin{enumerate}
    \item \label{ques:1} Which groups $\Comm(\Gamma)$ from Proposition \ref{prop:compute} contain a copy of a non-residually finite Baumslag--Solitar group?
    \item \label{ques:2} Does there exist a torsion-free group $G$ that embeds inside $\Comm(F_2)$ but not inside $\Comm(\Gamma_2)$? Vice versa?
    \item \label{ques:3} Let $\Gamma$ be a surface group. Do finitely generated subgroups of $\Comm(\Gamma)$ satisfy a Tits' alternative?
\end{enumerate}
 Note that a partial answer to Question \ref{ques:3} is given in the compact surface group case by \cite{MR1797749}, using the fact that $\Comm(\Gamma)$ embeds in $\Homeo(S^1)$ \cite{MR2115078}.

\appendix
\section{The GAP System}
In the appendices that follow we use the GAP System to complete our proofs.
Borrowing words from the creators \cite{GAP4}:
``GAP is a system for computational discrete algebra, with particular emphasis on Computational Group Theory. GAP provides a programming language, a library of thousands of functions implementing algebraic algorithms written in the GAP language as well as large data libraries of algebraic objects.'' GAP is especially well-suited for our needs, with key functions that we will use to explicitly evaluate commensurators.
Here is a short glossary of some of the key functions used in our code:

\vskip.1in
\paragraph{{\bf GroupHomomorphismByImages}( domain, codomain, list1, list2 )} Inputs two groups ``domain'' and ``codomain'' with lists ``list1'' and ``list2''. 
When GAP runs this command it first verifies that a well-defined homomorphism from ``domain'' to ``codomain'' sending ``list1'' to ``list2'' exists, returning ``fail'' otherwise.
If the homomorphism requested is well-defined, this returns a homomorphism from domain to codomain where elements of ``list1'' consisting of generators of domain are sent to corresponding elements in ``list2''.

\vskip.1in 
\paragraph{{\bf Image}(map, elem)} Inputs a homomorphism ``map'' and an element from the domain ``elem''. Returns the image of element ``elem'' under an application of the homomorphism ``map''.
\vskip.1in
\paragraph{{\bf InverseGeneralMapping}(map)} Inputs a homomomorphism ``map''. This function returns an inverse of an isomorphism of groups where the domain and codomain are not equal.
\vskip.1in
\paragraph{{\bf IsomorphismSimplifiedFpGroup}(G)} Inputs a finitely presented group ``G'', for which GAP applies Tietze transformations to a copy in order to reduce it with respect to the number of generators, the number of relators, and the relator lengths. When we apply this to a finite-index subgroup of an oriented cocompact surface group we get a one-relator group, as expected.
This function returns an isomorphism with domain G, codomain a group H isomorphic to G, so that the presentation of H has been simplified using Tietze transformations.
\vskip.1in
\paragraph{{\bf IsOne}(elem)} Inputs an element ``elem'' from a group. Returns true if ``elem'' is equal to the identity, and false if ``elem'' is not. This function is not guaranteed to terminate.

\section{Computer-assistance for the free group case}

\label{app:freegroup}

The code in \S\ref{ssec:codefree}  concretely defines maps $\phi$ and $\psi$ from the proof of the free group case.
Here K1 and K2 are the subgroups $\Delta_1$ and $\Delta_2$, and K1.1 corresponds to the element $A^{-2}$ in the free group $\left< A, B \right>$.

Running this code outputs ``false''. The GAP code verifies that each isomorphism exists and that the domains and ranges are all appropriate. In the end the GAP code computes the function
$$
w = \phi^{-1} \psi \phi \psi^{-1} \phi^{-1} \psi \phi \psi^{-1} \phi^{-1}
$$
with input Word = $B A B^{-1} A^{-1}$. It outputs Word10 = $A^{3} B A B^{-1} A^{2}$, and verifies that these are not equal (although this last part is easily done by hand).

After running the code below one can investigate properties of the map $\psi$. For instance, here $\psi$ is evaluated on $A^{-2}$:

\begin{verbatim}
gap> K1.1;
A^-2
gap> Image(psi, K1.1);
A^-3
\end{verbatim}
\noindent
The variable K1.1, corresponding to the element $A^{-2}$ in $F_2$, is shown here as being mapped to $A^{-3}$.

Moreover, one can verify that psi is an isomorphism K1 $\to$ K2 in GAP by checking that the map is surjective and that K1 and K2 have the same index in f:
\begin{verbatim}
gap> Image(psi, K1) = K2;
true
gap> Index(f, K1) = Index(f, K2);
true
\end{verbatim}

\subsection{The code} \label{ssec:codefree}
We ran the following code on GAP version 4.8.8. The usefulness of the code below rests on the identifications of the generator K1.1 with $A^{-2}$ and the generator K2.2 with $A^{3}$. Be aware that other versions of GAP may have different implementations of the function Kernel, in which $A^2$ and $A^3$ may correspond to different generators, or may not even be part of the chosen free generating set.

\begin{lstlisting}[language=GAP]
# Define the groups
f := FreeGroup("A", "B" );;
A := DirectProduct(CyclicGroup(2), CyclicGroup(3));;

# Define conjugation map, phi:
phi := GroupHomomorphismByImages ( f, f, [f.1, f.2], [f.1, f.1*f.2*f.1^(-1)]);;
phi2 := Inverse(phi);;

# Define the projection maps pi1 and pi2:
pi1 := GroupHomomorphismByImages( f, A, [f.1, f.2], [A.1, A.2]);;
pi2 := GroupHomomorphismByImages( f, A, [f.1, f.2], [A.2, A.1]);;

# Running Rank ensures K1 and K2 are equipped with finite presentations
K1:= Kernel(pi1);; Rank(K1);;
K2:= Kernel(pi2);; Rank(K2);;

# Define the map psi, there is a great deal of flexibility here, except that
# we want K1.1 (A^(-2)) to map to K2.2^(-1) (A^(-3)).
psi := GroupHomomorphismByImages ( K1, K2, 
    [K1.1, K1.2, K1.3, K1.4, K1.5, K1.6, K1.7],
    [K2.2^(-1), K2.5, K2.3, K2.4, K2.6, K2.1, K2.7]);;
psi2 := InverseGeneralMapping(psi);;

# Evaluate the word w in the residual finiteness kernel of BS(2,3):
Word := K1.2;; Word2 := Image(phi2, Word);;
Word3 := Image(psi2, Word2);; Word4 := Image(phi, Word3);;
Word5 := Image(psi, Word4);; Word6 := Image(phi2, Word5);;
Word7 := Image(psi2, Word6);; Word8 := Image(phi, Word7);;
Word9 := Image(psi, Word8);; Word10 := Image(phi2, Word9);;

#  Check to see if Word10 == Word. Outputs false, as desired:
IsOne(Word10*Word^(-1));
\end{lstlisting}

\section{Computer-assistance for the cocompact case}

\label{app:surfacegroup}

The code in \S\ref{ssec:cocompactcode} concretely defines maps $\phi$ and $\psi$ from the proof of the cocompact Fuchsian case.
Here K1 and K2 are the subgroups $\Delta_1$ and $\Delta_2$, and K1.1 corresponds to the element $A^{-2}$ in the surface group $g = \left< A, B, C, D : [A, B][C,D] \right>$.

In this case, more care must be taken in defining the map $\text{psi}: \text{K1} \to \text{K2}$ because K1 and K2 are not free groups. Explicitly defining maps from K1 to K2 using the GAP function GroupHomomorphismByImages usually results in maps that GAP cannot verify are well-defined isomorphisms (moreover, finding such maps from scratch is prohibitively difficult). To get around this obstruction, we first simplify the presentations of K1 and K2 before finding an isomorphism (akin to diagonalizing a matrix before doing computations).

We briefly explain the construction of psi in the code now. First, the code defines group homomorphisms Iso1 : K1 $\to$ fp1, Hom1 : fp1 
$\to$ Image1, Iso2 : K1 $\to$ fp2, and Hom2 : fp2 $\to$ Image2. GAP does not view K1 and K2 as finitely presented groups, and so the maps Iso1 and Iso2 are simply isomorphisms to their finite presentations (computed from the fact that K1 and K2 are finite-index normal subgroups). The maps Hom1 and Hom2 are maps to simplified finite presentations. The resulting maps from this construction are version dependent, so running our code in a different version of GAP may result in lastMap not being well-defined (in which case GAP will throw a fault). The images of the maps Hom1 and Hom2 have the following presentations. The red color indicates a part of the relation that is significantly different from the rest.

\begin{eqnarray*}
\text{Image1} &=& 
\langle F_1, F_2, F_3, F_4, F_5, F_6, F_7, F_8, F_{9}, F_{10}, F_{11}, F_{12}, F_{13}, F_{14}  : \\
&& [F_{11}, F_{12}] [F_4, F_5]{\color{red}F_1^{-1}F_6F_1} [F_{13},F_{14}] [F_9, F_{10}] {\color{red}F_6^{-1}} [F_7, F_8] [F_2, F_3] = 1 \rangle, \\
\text{Image2} &=& 
\langle F_1, F_2, F_3, F_4, F_5, F_6, F_7, F_8, F_{9}, F_{10}, F_{11}, F_{12}, F_{13}, F_{14}  : \\
&&[F_5, F_6] {\color{red}F_{4}} [F_{13}, F_{14}] [F_{7}, F_{8}] {\color{red} F_{1}^{-1} F_{4}^{-1} F_{1}} [F_{9}, F_{10}] [F_{2}, F_{3}][F_{11}, F_{12}]=1\rangle.
\end{eqnarray*}

The image of $A^{-2}$ under $\text{Hom1} \circ \text{Iso1}$ is $F_1$ in Image1. Moreover, $F_4$ maps to $A^3$ under $(\text{Hom2} \circ \text{Iso2})^{-1}$. Then lastMap : Image1 $\to$ Image2 is defined to be an isomorphism taking $F_1$ to $F_4$. The resulting map $\text{psi}: \text{K1} \to \text{K2}$ is then defined by the composition
$$
\text{psi} := \text{flipHom} \circ \text{Iso2}^{-1} \circ \text{Hom2}^{-1} \circ \text{lastMap} \circ \text{Hom1} \circ \text{Iso1},
$$
where flipHom sends $A$ to $A^{-1}$, correcting that lastMap sends $A^{-2}$ to $A^3$.

After verifying that the map psi is well-defined, the computer uses the finite presentation to verify that the input and output of the function
$$
w = \phi^{-1} \psi \phi \psi^{-1} \phi^{-1} \psi \phi \psi^{-1} \phi^{-1}
$$
\noindent
are not identical when evaluated at K1.2 = C in $g$ (one can also do this last check by hand using Dehn's algorithm).
The output to this evaluation, Word10, is:
\begin{verbatim}
A^-1*B^2*A^-1*B^-2*(A^-1*C*D*C^-1*D^-1*A*B)^2*A*B^-2*A^3*B*A^-1
*D*C*D^-1*C^-1*A*B^-1*A^-1*D*C*D^-1*C^-1*A^-2*B^2*A*B^-1*A^-1*D
*C*D^-1*C*D*C^-1*D^-1*A*B*A^-1*B^-2*A^2*C*D*C^-1*D^-1*A*B*A^-1
*C*D*C^-1*D^-1*A*B^-1*A^-3*B^2*A^-1*(B^-1*A^-1*D*C*D^-1*C^-1*A)^2
*B^2*A*B^-2*A
\end{verbatim}

As before, K1.1 corresponds to $A^{-2}$ and applying $\psi$ to it yields $A^{-3}$, as desired:

\begin{verbatim}
gap> K1.1; 
A^-2
gap> Image(psi, K1.1);
A^-3
\end{verbatim}
Moreover, one can verify that psi is an isomorphism K1 $\to$ K2 in GAP by checking that the map is surjective and that K1 and K2 have the same index in g:
\begin{verbatim}
gap> Image(psi, K1) = K2;
true
gap> Index(g, K1) = Index(g, K2);
true
\end{verbatim}

\subsection{The code} \label{ssec:cocompactcode}
We ran the following code on GAP version 4.8.8.
Note again that different versions of GAP may have different implementations of key functions, such as IsomorphismFpGroup and IsomorphismSimplifiedFpGroup, which are used below to find explicit presentations of $\Delta_1$ and $\Delta_2$. 
Because the definition of $\psi$ uses the presentations output by these functions, the definition below may not determine an isomorphism in other versions of GAP.

\begin{lstlisting}[language=GAP]
# Define the groups:
f := FreeGroup("A", "B", "C", "D" );;
comm := function (a,b) return a*b*a^(-1)*b^(-1);; end;;
g := f / [ comm(f.1,f.2)*comm(f.3,f.4) ];;
A := DirectProduct(CyclicGroup(2), CyclicGroup(3));;

# Setup for defining the map psi.
# Note that g.1, g.2, g.3, g.4 correspond to the images of A, B, C, D in f.
Pi1 := GroupHomomorphismByImages ( g, A, [g.1,g.2,g.3,g.4], 
    [A.1, A.2, A.1^0, A.1^0]);;
Pi2 := GroupHomomorphismByImages ( g, A, [g.1,g.2,g.3,g.4], 
    [A.2, A.1, A.1^0, A.1^0]);;

K1:= Kernel(Pi1);;
K2:= Kernel(Pi2);;

# Simplify the presentations of K1 and K2, while keeping track of
# the maps from K1 and K2 to their reduced presentations
Iso1 := IsomorphismFpGroup( K1 );;
fp1 := Image(Iso1);;
Hom1 := IsomorphismSimplifiedFpGroup(fp1);;
Image1 := Image(Hom1);;

Iso2 := IsomorphismFpGroup( K2 );;
fp2 := Image(Iso2);;
Hom2 := IsomorphismSimplifiedFpGroup(fp2);;
Image2 := Image(Hom2);;

# To define the map psi from K1 to K2, first define a map, lastMap,
# between the reduced presentations Image1 and Image2 of K1 and K2, 
# respectively.

# Conjugation map for defining lastMap
conj := ConjugatorAutomorphism(Image1, Image1.1^(-1)*Image1.6);;
apc := function (a) return Image(conj, a);; end;;

# A map to swap g.1 with g.1^{-1} so we get an image of BS(2,3) 
# and not BS(-2,3)
flipHom := GroupHomomorphismByImages( g, g, [g.1, g.2, g.3, g.4], 
    [g.1^(-1), g.2^(-1), g.2^(-1)*g.1^(-1)*g.3*g.1*g.2,
    g.2^(-1)*g.1^(-1)*g.4*g.1*g.2]);;

# This defines an isomorphism between reduced presentations of 
# K1 and K2.
lastMap := GroupHomomorphismByImages ( Image1, Image2,
    [Image1.1,apc(Image1.2),apc(Image1.3),apc(Image1.4),apc(Image1.5),
    Image1.6,apc(Image1.7),apc(Image1.8),Image1.9,Image1.10,
    apc(Image1.11), apc(Image1.12),Image1.13,Image1.14], 
    [Image2.4,Image2.2,Image2.3,Image2.5,Image2.6,Image2.1,
    Image2.9,Image2.10,Image2.7,Image2.8,Image2.11,Image2.12,
    Image2.13,Image2.14]);;

# Define the map psi:
psi := CompositionMapping(flipHom, InverseGeneralMapping(Iso2), 
    InverseGeneralMapping(Hom2), lastMap, Hom1, Iso1);;
psi2 := InverseGeneralMapping(psi);;

# Define conjugation map:
phi := GroupHomomorphismByImages ( g, g, [g.1, g.2, g.3, g.4], 
    [g.1, g.1*g.2*g.1^(-1), g.1*g.3*g.1^(-1), g.1*g.4*g.1^(-1)]);;
phi2 := Inverse(phi);;

# Evaluate the word w in the residual finiteness kernel of BS(2,3):
Word := K1.2;; Word2 := Image(phi2, Word);;
Word3 := Image(psi2, Word2);; Word4 := Image(phi, Word3);;
Word5 := Image(psi, Word4);; Word6 := Image(phi2, Word5);;
Word7 := Image(psi2, Word6);; Word8 := Image(phi, Word7);;
Word9 := Image(psi, Word8);; Word10 := Image(phi2, Word9);;

# Check to see if Word10 == Word. Outputs false, as desired:
IsOne(Word10*Word^(-1));
\end{lstlisting}

\bibliography{refs}
\bibliographystyle{amsalpha}

\end{document}